\font \sevenrm=cmr7
\font \fiverm=cmr5
\documentclass[12pt, english]{amsart}
\input epsf.sty
\usepackage{amsfonts}
\usepackage{amssymb}
\usepackage{amsmath}
\usepackage{amsthm}
\usepackage{amscd}
\usepackage[all]{xy} 
\usepackage{epsfig,exscale}
\usepackage{color}
\usepackage{graphicx}
\usepackage{xspace}
\usepackage{axodraw4j}
\usepackage{colordvi}

\newcommand{\nc}{\newcommand}

{\everymath{\displaystyle\everymath{}}\array}%
{\endarray}
\newtheorem{theorem}{Theorem}
\newtheorem{definition}{Definition}

\newtheorem{proposition}{Proposition}
\newtheorem{ex}{Example}
\newtheorem{remark}{Remark}
\nc{\comment}[1]{[[{\tt #1}]] }
\nc{\Cal}[1]{{\mathcal {#1}}}
\nc{\mop}[1]{\mathop{\hbox {\rm #1} }\nolimits}
\nc{\gmop}[1]{\mathop{\hbox {\bf #1} }\nolimits}

\nc{\smop}[1]{\mathop{\hbox {\sevenrm #1} }\nolimits}
\nc{\ssmop}[1]{\mathop{\hbox {\fiverm #1} }\nolimits}
\nc{\mopl}[1]{\mathop{\hbox {\rm #1} }\limits}
\def\dbar{d\hskip-3pt \raise 4pt\hbox{-}}
\nc{\smopl}[1]{\mathop{\hbox {\sevenrm #1} }\limits}
\nc{\ssmopl}[1]{\mathop{\hbox {\fiverm #1} }\limits}
\nc{\frakg}{{\frak g}}
\nc{\g}[1]{{\frak {#1}}}
\def \restr#1{\mathstrut_{\textstyle |}\raise-6pt\hbox{$\scriptstyle #1$}}
\def \srestr#1{\mathstrut_{\scriptstyle |}\hbox to
-1.5pt{}\raise-4pt\hbox{$\scriptscriptstyle #1$}}
\nc{\wt}{\widetilde} \nc{\wh}{\widehat}
\nc{\redtext}[1]{\textcolor{red}{#1}}
\nc{\bluetext}[1]{\textcolor{blue}{#1}}
\nc\fleche[1]{\mathop{\hbox to #1 mm{\rightarrowfill}}\limits}
\nc{\ignore}[1]{}
\def\semi{\mathrel{\times}\kern -.85pt\joinrel\mathrel{\raise
1.4pt\hbox{${\scriptscriptstyle |}$}}}
\nc\R{{\mathbb R}}
\nc\N{{\mathbb N}}
\nc\inver{^{-1}}
\nc\point{\hbox{\bf .}}
\nc\un{\hbox{\bf 1}}

\def\racine{{\scalebox{0.3}{ 
\begin{picture}(12,12)(38,-38)
\SetWidth{0.5} \SetColor{Black} \Vertex(45,-33){5.66}
\end{picture}}}}

\def\arbreca{\,{\scalebox{0.15}{
\begin{picture}(8,156) (370,-147)
\SetWidth{2}
\SetColor{Black}
\Line(374,-143)(374,-99)
\Vertex(374,-96){9}
\Vertex(375,-144){12}
\Line(374,-92)(374,-48)
\Vertex(374,-45){9}
\Line(374,-42)(374,2)
\Vertex(374,5){9}
\end{picture}
}}\,}

\def\arbredh{\,{\scalebox{0.15}{
\begin{picture}(90,46) (330,-257)
\SetWidth{2}
\SetColor{Black}
\Vertex(375,-254){12}
\Line(376,-252)(395,-217)
\Vertex(395,-215){9}
\Line(374,-254)(335,-226)
\Vertex(334,-224){9}
\Line(375,-252)(356,-215)
\Vertex(355,-215){9}
\Line(374,-255)(417,-227)
\Vertex(418,-225){9}
\end{picture}
}}\,}

\def\racine{\,{\scalebox{0.07}{
\begin{picture}(29,29) (360,-285)
    \SetWidth{6}
    \SetColor{Black}
    \Vertex(375,-271){20}
  \end{picture}
  }}\,}
\def\echela{\,{\scalebox{0.07}{
\begin{picture}(33,116) (353,-443)
    \SetWidth{6}
    \SetColor{Black}
    \Vertex(369,-428){20}
    \Vertex(369,-341){16}
    \Line(369,-341)(369,-419)
  \end{picture}
  }}\,}
  \def\echelb{\,{\scalebox{0.07}{
   \begin{picture}(33,195) (351,-363)
    \SetWidth{6}
    \SetColor{Black}
    \Vertex(369,-262){16}
    \Line(369,-262)(369,-340)
    \Vertex(368,-348){20}
    \Line(369,-184)(369,-262)
    \Vertex(369,-182){16}
  \end{picture}
  }}\,}
  \def\arbrey{\,{\scalebox{0.07}{
  \begin{picture}(147,114) (299,-444)
    \SetWidth{6}
    \SetColor{Black}
    \Vertex(368,-429){20}
    \Vertex(313,-344){16}
    \Vertex(434,-348){16}
    \Line(312,-344)(368,-428)
    \Line(433,-348)(372,-428)
  \end{picture}
}}\,}
\def\arbreza{\,{\scalebox{0.07}{
\begin{picture}(147,200) (299,-358)
    \SetWidth{6}
    \SetColor{Black}
    \Vertex(368,-343){20}
    \Vertex(313,-258){16}
    \Vertex(434,-262){16}
    \Line(312,-258)(368,-342)
    \Line(433,-262)(372,-342)
    \Vertex(313,-172){16}
    \Line(311,-179)(311,-252)
  \end{picture}
}}\,}
  \def\arbrema{\,{\scalebox{0.07}{
   \begin{picture}(149,215) (296,-345)
    \SetWidth{6}
    \SetColor{Black}
    \Line(370,-244)(370,-317)
    \Vertex(371,-238){16}
    \Line(310,-149)(366,-233)
    \Line(432,-150)(372,-234)
    \Vertex(310,-144){16}
    \Vertex(433,-144){16}
    \Vertex(371,-330){20}
  \end{picture}
  }}\,}
\def\diagramme #1{\vskip 4mm \centerline {#1} \vskip 4mm}
\begin{document}
\title{
{Doubling bialgebras of rooted trees}}        
\author{Mohamed Belhaj Mohamed}
\address{{Laboratoire de math\'ematiques physique fonctions sp\'eciales et applications, universit\'e de sousse, rue Lamine Abassi 4011 H. Sousse,  Tunisie}}     
         \email{mohamed.belhajmohamed@isimg.tn}
 
\author{Dominique Manchon}
\address{Universit\'e Blaise Pascal,
       C.N.R.S.-UMR 6620,
        63177 Aubi\`ere, France}       
        \email{manchon@math.univ-bpclermont.fr}
        \urladdr{http://math.univ-bpclermont.fr/~manchon/}     
\date{May 2016}
\noindent{\footnotesize{${}\phantom{a}$ }}
\begin{abstract}
The linear space of rooted forest admits two graded bialgebra structures. The first is defined by A. Connes and D. Kreimer using admissible cuts, and the second is defined by D. Calaque, K. Ebrahimi-Fard and the second author using contraction of trees. In this article we define the doubling  of these two spaces. We construct two bialgebra  structures on these space which are in interaction, as well as two associative products.
We also show that these two bialgebras verify a commutative diagram similar to the diagram verified D. Calaque, K. Ebrahimi-Fard and the second author in the case of rooted trees Hopf algebra, and by the second author in the case of cycle free oriented graphs oriented graphs.
\end{abstract}
\maketitle

\tableofcontents

\section{Introduction}
 Rooted trees appear in the work of Cayley \cite{ca} in the context of differential equations. They are used in an essential way in the work of Butcher \cite{bu}, Grossman and Larson \cite{gl}, Munthe-Kaas and Wright \cite{mw} in the field of numerical analysis. They also appear in the context of renormalization in perturbative quantum field theory in the works of A. Connes and D. Kreimer \cite{ad98, A.D2000, DK98}, D. Calaque, K. Ebrahimi-Fard, D. Manchon \cite{ckm, ms} and L. Foissy \cite{lf}.\\

On the vector space $\Cal H$ spanned by the rooted forests and graded by the number of vertices, A. Connes and D. Kreimer introduce a Hopf algebra structure where the coproduct is defined by:
\begin{eqnarray*}
\Delta_{CK}(t) &=& \sum_{c \in {\mop{\tiny{Adm}}(t)}} P^c(t)\otimes R^c(t),
\end{eqnarray*}
where $\mop{Adm}(t)$ is the set of admissible cuts.\\
In the same context D. Calaque, K. Ebrahimi-Fard and the second author introduce on the commutative algebra $\tilde{\Cal H}$ generated by rooted forests a structure of bialgebra graded by the number of edges. The coproduct is defined by:  
$$\Delta_{\tilde {\Cal H}} (t) = \sum_{s\subseteq t} s \otimes t/s,$$
where $s$ is a covering subforest of the rooted tree $t$ and $t / s$ is the tree obtained by contracting each connected component of $s$ onto a vertex. Also they establish a relation between the bialgebra $\tilde{\Cal H}$ obtained this way and the Connes-Kreimer Hopf algebra of rooted trees $\Cal H$ by means of a natural $\tilde{\Cal H}$-comodule structure on $\Cal H$ given by:
$$\Phi(t)=\Delta_{\tilde {\Cal H}} (t) = \sum_{s\subseteq t} s \otimes t/s.$$
To be precise, the following diagram commutes:
\diagramme{
\xymatrix{
\Cal H \ar[d]_{\Delta} \ar[rr]^{\Phi} 
& &\tilde {\Cal H} \otimes\Cal H \ar[dd]^{I \otimes \Delta }\\
\Cal H \otimes\Cal H \ar[d]_{\Phi\otimes\Phi}\\
\tilde {\Cal H} \otimes\Cal H\otimes \tilde {\Cal H} \otimes \Cal H \ar[rr]^{m^{13}} 
&& \tilde {\Cal H}\otimes\Cal H\otimes \Cal H, }
}
\noindent
making $\Cal H$ a comodule-bialgebra on $\tilde{\Cal H}$ \cite{ckm}.\\

In this paper, we define the doubling spaces of $\Cal H$ and $\tilde{\Cal H}$ respectively denoted by $D$ and $\tilde{D}$. In other words, we denoted by ${V}$ the vector space spanned by the couples $(t,s)$ where $t$ is a tree and $s =P^{c_0}(t)$ where $c_0$ is an admissible cut of $t$. The doubling space ${D}$ is the symmetric algebra of $V$, i.e: ${D} : = S({V})$. Similarly, if $\tilde V$ is the vector space spanned by the couples $(t,s)$ where $t$ is a tree, and $s$ is a subforest of $t$, the doubling space $\tilde{D}$ is the symmetric algebra of $\tilde V$, i.e: $\tilde D : = S(\tilde V)$. Note that $D$ is strictly included in $\tilde{D}$, as there are subforest $s$ which are not of the form $P^{c_0}(t)$.\\ 
We prove that there exist graded bialgebra structures on $D$ and $\tilde{D}$, where the coproducts are defined respectively by:
For all $(t,s) \in D$: 
$$\Delta(t,s) =  \sum_{c \in {\mop{\tiny{Adm}}(s)}} \big(t, P^c(s)\big)\otimes \big( R^c(t), R^c(s)\big),$$
and for all  $(t,s) \in \tilde D$:
$$\Gamma (t,s) = \sum_{s'\subseteq s} (t , s') \otimes(t/s', s/s').$$ 
We show that $\Delta(V) \subset V\otimes V$ and $\Gamma (\tilde {V}) \subset \tilde {V}\otimes \tilde {V}$, which allows us to restrict $\Delta$ on $V$ and $\Gamma$ on $\tilde {V}$.\\
In the second part of this paper, we prove that $D$ admits a comodule structure on $\tilde D$ given by the coaction $\phi: D \longrightarrow \tilde{D} \otimes D$, which is defined for all $(t,s) \in D$ by restriction of $\Gamma$ to $D$: 
$$\phi (t,s) = \sum_{s'\subseteq s} (t , s') \otimes(t/s', s/s').$$
The coaction $\phi$ can also restrict to $V$ because $\phi (V) \subset \tilde{V}\otimes V$.\\
We construct an associative algebra structure on $V$ given by the associative product $\circledast : V \otimes V \longrightarrow V$, defined for all couples of forests $(t,s), (t',s')$ such that $s = P^{c} (t)$ and $s' = P^{c'} (t')$ by:
\begin{equation}
(t,s) \circledast (t',s') = \left\lbrace
\begin{array}{lcl}
(t,s \cup s')  \;\;\;\;\; \text{if} \;\; t' = R^c (t)\\
0  \;\; \;\; \;\; \;\; \text{if not},
\end{array}\right. 
\end{equation}
where $s \cup s'$ is the pruning of the cut $c'$ raised to the tree $t$. This product is obtained by dualizing the restriction of the coproduct $\Delta$ to $V$, identifying $V$ with its graded dual using the basis $\{(t,P^c(t)),\, t \hbox{ rooted tree and }c\in\mop{Adm}(t)\}$. We accordingly construct a second associative algebra structure on $\tilde V$ by dualizing the restriction of the coproduct $\Gamma$ to $\tilde V$, yielding the associative product $\sharp : \tilde V \otimes \tilde V \longrightarrow \tilde V$, defined by:
\begin{equation}
(t,s) \sharp (t',s') = \left\lbrace
\begin{array}{lcl}
(t,s \cup s')  \;\;\;\;\; \text{if} \;\; t' = t/s\\
0  \;\; \;\; \;\; \;\; \text{if not}.
\end{array}\right. 
\end{equation}
In the end of this article, we define a new map $\xi : \tilde V \otimes \tilde V \otimes V \longrightarrow \tilde V \otimes V$ by:
\begin{enumerate}
\item $\xi \big( (t', s') \otimes (t'',s'')\otimes (u, v)\big) = (t', s' \cup s'')\otimes (t'/(s' \cup s''), v),$\\
if $t'' = R^c(t')$, $v = P^{\tilde{c}}(t')$ and $u = t'/s'$, where $c$ is an admissible cut of $t'$, and $\tilde{c}$ is an admissible cut of $t'$ which does not meet $s'$ and $s''$.
\item $\xi \big( (t', s') \otimes (t'',s'')\otimes (u, v)\big) = 0,$\\ 
if $t', t'', s', s'', u$ and $v$ are forests which do not match the conditions of item (1). 
\end{enumerate} 
We prove that the coaction $\phi$ and the map $\xi$ make the following diagram commute:
\diagramme{
\xymatrix{
V \ar[d]_{\Delta} \ar[rr]^{\phi} 
&&\tilde V \otimes V  \ar[d]^{id \otimes \Delta}\\
 V \otimes V \ar[d]_{\phi\otimes \phi}&&\tilde V \otimes V \otimes V \\
\tilde V \otimes V \otimes \tilde V \otimes V \ar[rr]_{\tau^{23}} 
&&\ar[u]_{\xi \otimes id } \tilde V \otimes \tilde V \otimes  V \otimes V
 }
}
Moreover this diagram extends to the diagram:
\diagramme{
\xymatrix{
D \ar[d]_{\Delta} \ar[rr]^{\phi} 
&&\tilde D \otimes D  \ar[d]^{id \otimes \Delta}\\
 D \otimes D \ar[d]_{\phi\otimes \phi}&&\tilde D \otimes D \otimes D \\
\tilde D \otimes D \otimes \tilde D \otimes D \ar[rr]_{\tau^{23}} 
&&\ar[u]_{\xi \otimes id } \tilde D \otimes \tilde D \otimes  D \otimes D
 }
}
where the arrows are now algebra morphisms. This second diagram is similar to the commutative diagram verified by D. Calaque, K. Ebrahimi-fard and the second author in the case of rooted trees hopf algebra and by the second author in the case of cycle-free oriented graphs oriented graphs. The only difference is that the map $m^{13}$ is replaced here by the map $(\xi \otimes id)\circ \tau^{23}$. 
\section{Hopf algebras of rooted trees}
A {\sl rooted tree\/} is a finite connected simply connected oriented graph such that every vertex has exactly one incoming edge, except for a distinguished vertex (the root) which has no incoming edge.
The set of rooted trees is denoted by $T$ and the set of rooted trees with $n$ vertices is denoted by $T_n$.   
\begin{ex}
\begin{eqnarray*}
T_1 &=& \{\racine\}\\
T_2 &=& \{\echela\}\\
T_3 &=& \{\echelb , \arbrey \}\\
T_4 &=& \{\arbreca, \arbrema, \arbreza, \arbredh\}
\end{eqnarray*}
\end{ex}
Let $\mathcal{H} = S({T})$ be the algebra of rooted forest. A. Connes and D. Kreimer \cite{ad98}, \cite{DK98} showed that this space, graded according to the number of vertices, admits a structure of graded bialgebra. The product is the concatenation, and the coproduct is defined by:
\begin{eqnarray*}
\Delta_{CK}(t) &=& t \otimes \un + \un \otimes t + \sum_{c \in {\mop{\tiny{Adm'}}(t)}} P^c(t)\otimes R^c(t)\\
&=&\sum_{c \in {\mop{\tiny{Adm}}(t)}} P^c(t)\otimes R^c(t),
\end{eqnarray*}
where $\mop{Adm}(t)$ (resp $\mop{Adm'}(t)$) is the set of admissible cuts (resp. nontrivial admissible cuts) of a forest, i.e. the set of collections of edges such that any path from the root to a leaf contains at most one edge of the collection. We denote as usual by $P^c(t)$ (resp. $R^c(t)$) the pruning (resp. the trunk) of $t$, i.e. the subforest formed by the edges above the cut $c \in \mop{Adm}(t)$ (resp. the subforest formed by the edges under the cut). Note that the trunk of a tree is a tree, but the pruning of a tree may be a forest. $\un$ stands for the empty forest, which is the unit. One sees easily that $\deg(t) = \deg(P_c(t)) + \deg(R_c(t))$ for all admissible cuts. (See \cite{ckm} and \cite{lf}).\\
D. Calaque, K. Ebrahimi-Fard and D. Manchon showed that the space $\tilde {\Cal H}$ generated by the rooted forests, graded according to the number of edges, admits a structure of graded bialgebra \cite{ckm}. The unit is the empty forests, the product is the concatenation, and the coproduct is defined for any non empty forest $t$ by:
 $$\Delta_{\tilde {\Cal H}} (t) = \sum_{s\subseteq t} s \otimes t/s,$$
where $s$ is a covering subforest of a rooted tree $t$ and $t / s$ the tree obtained by contracting each connected component of $s$ onto a vertex, i.e. $s$ is a collection of disjoint sub-trees $(t_1, \cdots ,t_n)$ of $t$, covering $t$. In particular, two sub-trees of the forest have no vertex in common. The tree $t/s$ is obtained by contraction of each connected component of $s$ on a vertex.
\begin{ex}
The two coproducts applied to same tree $\arbreza$:
$$
\Delta_{\Cal H} (\arbreza)=\racine \racine \racine \racine \otimes\arbreza+\arbreza\otimes\racine+2\echela \racine \racine \otimes\arbrey+\echela \racine \racine\otimes\echelb+\echelb \racine\otimes\echela+\arbrey\racine\otimes\echela+\echela\echela\otimes\echela.
$$
$$
\Delta_{CK}(\arbreza)=\textbf{1}\otimes\arbreza+\arbreza\otimes
\textbf{1}+\racine\otimes\echelb+\echela\otimes\echela+\racine\otimes\arbrey+\echela\racine\otimes\racine
+\racine\racine\otimes\echela
$$
\end{ex}
The Hopf algebra $\Cal H'$ is given by identifying all elements of degree zero to unit $\un$:
\begin{equation}
{\Cal H'} = \wt{\Cal H} / \Cal J
\end{equation}
where $\Cal J$ is the ideal generated by the elements $\un - t$ where $t$ is a forest of degree zero.\\
The example of coproduct above becomes by identifying the unit to $\racine$:
$$
\Delta_{\Cal H'} (\arbreza)=\racine \otimes\arbreza+\arbreza\otimes\racine+2\echela \otimes\arbrey+\echela\otimes\echelb+\echelb \otimes\echela+\arbrey\otimes\echela+\echela\echela\otimes\echela.
$$

\section{Doubling bialgebras of trees}
We have studied the concept of doubling bialgebra in the context of the specified Feynman graphs Hopf algebra \cite{MBM}. We have proved that the doubling space of specified Feynman graphs, the vector space spanned by the $(\bar\Gamma, \bar\gamma)$ where $\bar\Gamma$ is locally $1PI$ specifed graph of the theory $\Cal T$, $\bar\gamma \subset \bar\Gamma$ loc $1PI$ and $\bar\Gamma / \bar\gamma$ is a specifed graph of $\Cal T$, admits a structure of graded bialgebra (see \cite{MBM} and \cite[\S 3]{mbm}). 

\subsection{Doubling bialgebra $\mathcal{H}_{CK}$}
Let ${V}$ the vector space spanned by the couple $(t,s)$ where $t$ is a tree and $s =P^{c_0}(t)$ where $c_0$ is an admissible cut of $t$. We define then the doubling bialgebra of trees $\mathcal{H}_{CK}$ by ${D} : = S({V})$ and we define the coproduct $\Delta$ for all $(t,s) \in D$ by: 
$$\Delta(t,s) =  \sum_{c \in {\mop{\tiny{Adm}}(s)}} \big(t, P^c(s)\big)\otimes \big( R^c(t), R^c(s)\big).$$
\begin{theorem}
$D$ is a graded bialgebra.
\end{theorem}
\begin{proof}
The unit $\textbf{1}$ is identified to empty forest, the counit $\varepsilon$ is given by $\varepsilon (t, s) = \varepsilon (s)$ and the graduation is given by the number of vertices of $s$:
$$ |(t, s)|  = |s|.$$
The product is defined by:
$$(t, s)(t', s') = (tt', ss').$$
We now calculate:
\begin{eqnarray*}
(\Delta \otimes id)\Delta (t, s) &=& (\Delta \otimes id) \Big( \sum_{c \in {\mop{\tiny{Adm}}(s)}} \big(t, P^c(s)\big)\otimes \big( R^c(t), R^c(s)\big) \Big)\\
&=& \sum_{c \in {\mop{\tiny{Adm}}(s)}\atop c' \in {\mop{\tiny{Adm}}(P^c(s))}} \Big(t,P^{c'}\big(P^c(s)\big) \Big) \otimes \Big( R^{c'}(t), R^{c'}\big(P^c(s)\big)\Big) \otimes \big( R^c(t), R^c(s)\big)\\
&=& \sum_{{c \in {\mop{\tiny{Adm}}(s)};c' \in {\mop{\tiny{Adm}}(s)}}\atop {c'> c }} \big(t,P^{c'}(s)\big) \otimes \Big( R^{c'}(t), R^{c'}\big(P^c(s)\big)\Big) \otimes \big( R^c(t), R^c(s)\big).
\end{eqnarray*}
The notation ${c '> c}$ denotes the cut $c$ is below the cut $c'$. \\
On the other hand,
\begin{eqnarray*}
(id \otimes \Delta)\Delta (t, s) &=& (id \otimes \Delta) \Big( \sum_{c' \in {\mop{\tiny{Adm}}(s)}} \big(t, P^{c'}(s)\big)\otimes \big( R^{c'}(t), R^{c'}(s)\big) \Big)\\
&=& \sum_{c' \in {\mop{\tiny{Adm}}(s)} \atop c \in {\mop{\tiny{Adm}}(R^{c'}(t))}} \big(t, P^{c'}(s)\big)\otimes \Big( R^{c'}(t), P^c \big(R^{c'}(s)\big)\Big)\otimes \Big(R^c( R^{c'}(t)\big), R^c \big( R^{c'}(s)\big)\Big).
\end{eqnarray*}
The condition $\{ c' \in {\mop{\tiny{Adm}}(s)} ; c \in {\mop{\tiny{Adm}}(R^{c'}(t))} \}$ is equivalent to $\{{c \in {\mop{\tiny{Adm}}(s)} ; c' \in {\mop{\tiny{Adm}}(s)}} \text{and}\; {c'> c }\}$ and we obtain the following equalities:
$$P^c (R^{c'}(s)) = R^{c'} (P^c (s)), \hskip 5mm R^c( R^{c'}(t)) = R^c(t), \hskip 5mm R^c( R^{c'}(t)) = R^c(s).$$
Then: 
\begin{eqnarray*}
(id \otimes \Delta)\Delta (t, s)&=& \sum_{{c \in {\mop{\tiny{Adm}}(s)};c' \in {\mop{\tiny{Adm}}(s)}}\atop {c'> c }} \big(t,P^{c'}(s)\big) \otimes \Big( R^{c'}(t), R^{c'}\big(P^c(s)\big)\Big) \otimes \big( R^c(t), R^c(s)\big)\\
&=& (\Delta \otimes id)\Delta (t, s).
\end{eqnarray*}
Hence $(\Delta \otimes id)\Delta = (id \otimes \Delta)\Delta$, and consequently $\Delta$ is co-associative. Finally we have directly:
$$\Delta \big((t, s).(t', s')\big) = \Delta (t, s) \Delta (t', s').$$ 
\end{proof}
\begin{remark}
We remark here that $\Delta (V) \subset V \otimes V$. Indeed, if $(t, s) \in V$ then $\big(t,P^{c}(s)\big) \in V$, since a pruning of $s$ is also a pruning of $t$. Similary $\big( R^c(t), R^c(s)\big) \in V$ because $R^c(t)$ is a tree, and $R^c(s) = R^c(P^{c_0} (t)) = P^{c_0} (R^c(t))$ is a pruning of $R^c(t)$. So we can restrict the coassociative product $\Delta$ to $V$.  
\end{remark}
\begin{proposition}\label{p1} The second projection 
\begin{eqnarray*}
P_2: D &\longrightarrow& {{\Cal H}} \\ (t, s) &\longmapsto& s
\end{eqnarray*} 
is a bialgebra morphism.
\end{proposition}
\begin{proof}
The fact that $P_2$ is an algebra morphism is trivial. It suffices to show that $P_2$ is a coalgebra morphism, i.e. $P_2$ verifies the following commutative diagram: 
\diagramme{
\xymatrix{ D \ar[rrr]^{P_2}\ar[d]_{\Delta}
&&&\Cal H\ar[d]^{\Delta}\\
D \otimes D
\ar[rrr]_{P_2 \otimes P_2}&&&\Cal H \otimes \Cal H
}
}
which can be seen by direct calculation:
\begin{eqnarray*}
\Delta \circ P_2 (t, s) &=& \Delta(s) \\  
&=&\sum_{c \in {\mop{\tiny{Adm}}(s)}} P^c(s)\otimes  R^c(s)\\
&=&\sum_{c \in {\mop{\tiny{Adm}}(s)}} P_2 \big(t, P^c(s)\big)\otimes P_2 \big( R^c(t), R^c(s)\big)\\
&=& (P_2 \otimes P_2) \Delta (t, s).
\end{eqnarray*}
\end{proof}
\subsection{Doubling bialgebra $\tilde{\mathcal H}$}
Let $\tilde V$ be the vector space spanned by the couples $(t,s)$ where $t$ is a tree, and $s$ is a subforest of $t$. We define then the doubling of bialgebra $\tilde {\Cal H}$ by $\tilde D : = S(\tilde V)$ and we define the coproduct $\Gamma$ for all  $(t,s) \in \tilde D$ by:
$$\Gamma (t,s) = \sum_{s'\subseteq s} (t , s') \otimes(t/s', s/s'),$$ 
\begin{theorem}
$\tilde D$ is a graded bialgebra.
\end{theorem}
\begin{proof}
The unit $\textbf{1}$ is identified to empty graph, the counit $\varepsilon$ is given by $\varepsilon (t, s ) = \varepsilon (s)$ and the graduation is given by the number of vertices of $s$:
$$ |(t, s)|  = |s|.$$
The product is given by:
$$(t, s)(t', s') = (tt', ss').$$
The coassocitivity of coproduct $\Gamma$ is given by this calculation:
\begin{eqnarray*}
(\Gamma \otimes id)\Gamma (t, s) &=& (\Gamma \otimes id) \big( \sum_{s' \subseteq s } ( t, s')  \otimes (t / s, s / s') \big)\\
&=& \sum_{s' \subseteq s \atop s'' \subseteq s' } ( t, s'')\otimes (t / s'', s' / s'') \otimes (t / s, s / s')\\
&=& \sum_{s'' \subseteq s' \subseteq s } ( t, s'')\otimes (t / s'', s' / s'') \otimes (t / s, s / s'),
\end{eqnarray*}
whereas
\begin{eqnarray*}
(id \otimes \Gamma)\Gamma (t, s) &=& (id \otimes \Gamma) \big( \sum_{s'' \subseteq s } ( t, s'')  \otimes (t / s, s / s'') \big)\\
&=& \sum_{s'' \subseteq s  \atop r \subseteq s/s'' } ( t, s'')\otimes (t / s'', r) \otimes ( (t /s'')/r, (t /s'')/r ).
\end{eqnarray*}
As $r \subseteq s/s''$, then there exists a forest $s'$ such that: $s'' \subseteq s' \subseteq s$ and $r \cong s' / s''$. Hence:
\begin{eqnarray*}
(id \otimes \Gamma)\Gamma (t, s) &=&  \sum_{s'' \subseteq s' \subseteq s } ( t, s'')\otimes (t / s'', s' / s'') \otimes ((t /s'')/(s' /s''), (s /s'')/(s' /s''))\\
&=& \sum_{s'' \subseteq s' \subseteq s } ( t, s'')\otimes (t / s'', s' / s'') \otimes (t / s, s / s').
\end{eqnarray*}
Therefore $(\Gamma \otimes id)\Gamma = (id \otimes \Gamma)\Gamma$, and thus $\Gamma$ is coassociative. Finally we show immediately that:
$$\Gamma \big((t, s) (t', s')\big) = \Gamma (t, s)\Gamma (t', s').$$
\end{proof}
\begin{remark}
We note here that $\Gamma (\tilde V) \subset \tilde V \otimes \tilde V$. Indeed, if $(t, s) \in V$ then $(t, s') \in \tilde V$ and $(t/s', s/s') \in \tilde V$. So we can restrict the coassociative product $\Gamma$ to $\tilde V$.  
\end{remark}
\begin{proposition} The second projection 
\begin{eqnarray*}
P_2: {\wt{D}} &\longrightarrow& {\wt{\Cal H}} \\ (t, s) &\longmapsto& s
\end{eqnarray*} 
is a morphism of graded bialgebras.
\end{proposition}
\begin{proof}
The fact that $P_2$ is an algebras morphism is trivial, it suffices to show that $P_2$ is a coalgebra morphism, analogously to Proposition \ref{p1}:
\begin{eqnarray*}
\Gamma \circ P_2 (t, s) &=& \Gamma(s) \\  
&=& \sum_{s' \subseteq s} s'\otimes s' / s'\\ 
&=& \sum_{s' \subseteq s} P_2 (t, s')\otimes P_2 (t/s', s/s')\\ 
&=& (P_2 \otimes P_2) \Gamma (t, s).
\end{eqnarray*}
\end{proof}
\section{Comodule structure}
\subsection{Comodule structure on bialgebras of trees and bialgebras of oriented graphs}
D. Calaque, K. Ebrahimi-Fard and D. Manchon  have studied the Connes-Kreimer Hopf algebra $\Cal H$ graduated following the number of vertices in \cite{ckm}, as comodule on a Hopf algebra of rooted trees $\tilde {\Cal H}$ graduated according to the number of edges. This structure is defined as follows: For $\un$ we have: $\Phi(\un)=\racine \otimes \un$, and for any non-empty tree $t$ we have:
$$\Phi(t)=\Delta_{\tilde {\Cal H}} (t) = \sum_{s\subseteq t} s \otimes t/s.$$ We can also write $\Phi(t)$ as follows:
\begin{eqnarray*}
\Phi(t) &=& \Delta_{\tilde {\Cal H}} (t) = \sum_{s\subseteq t} s \otimes t/s\\
&=& \racine \otimes t + \left( t \otimes \racine  +  \sum_{s \hbox{ \sevenrm proper sub-forest of} t} s \otimes t/s \right ).
\end{eqnarray*}
D. Calaque, K. Ebrahimi-Fard and the second author showed the existence of a relation between this coaction $\Phi$ and the Connes-Kreimer coproduct $\Delta_{CK}$. 
\begin{equation}
\Delta_{CK}(t) = t \otimes \un + \un \otimes t + \sum_{c \in {\mop{\tiny{Adm}}(t)}} P^c(t)\otimes R^c(t),
 \end{equation}
 This relation is given by this theorem:
\begin{theorem}\cite{ckm} 
This diagram is commutative:
\diagramme{
\xymatrix{
\Cal H \ar[d]_{\Delta} \ar[rr]^{\Phi} 
& &\tilde {\Cal H} \otimes\Cal H \ar[dd]^{I \otimes \Delta_{CK} }\\
\Cal H \otimes\Cal H \ar[d]_{\Phi\otimes\Phi}\\
\tilde {\Cal H} \otimes\Cal H\otimes \tilde {\Cal H} \otimes \Cal H \ar[rr]^{m^{13}} 
&& \tilde {\Cal H}\otimes\Cal H\otimes \Cal H }
}
i.e : The following identity is verified:
\begin{equation}\label{codistrib}
(\mop{Id}_{\tilde {\Cal H}} \otimes\Delta_{CK})\circ\Phi=m^{1,3}\circ(\Phi\otimes\Phi)\circ\Delta_{CK},
\end{equation}
where $ m^{1,3}:\tilde {\Cal H} \otimes \Cal H \otimes \tilde {\Cal H} \otimes \Cal H 
\longrightarrow \tilde {\Cal H} \otimes \Cal H \otimes \Cal H $ is defined by:
\begin{equation}
m^{1,3}(a\otimes b\otimes c\otimes d)=ac\otimes b\otimes d.
\end{equation}
\end{theorem}
\subsection{Comodule structure on the doubling of the rooted trees bialgebra}
We define $\phi : D \longrightarrow \tilde{D} \otimes D$ for all $(t,s) \in D$ by:
$$\phi (t,s) = \sum_{s'\subseteq s} (t , s') \otimes(t/s', s/s').$$
The map $\phi$ is well defined. Indeed, if $(t,s) \in D$ i.e: $s = P^{c_0} (t)$ for an admissible cut $c_0$ of $t$, we have: 
$$s \subseteq t \Longrightarrow s' \subseteq s \subseteq t \Longrightarrow (t, s') \in \tilde D,$$
and $s/s' = P^{c} (t/s')$, where $c$ is the admissible cut deduced from $c_0$. Therefore $(t/s', s/s') \in D$.  
\begin{theorem}
$D$ admits a comodule structure  on $\tilde{D}$ given by $\phi$. 
\end{theorem}
\begin{proof}
The proof amounts to show that the following diagram is commutative:
\diagramme{
\xymatrix{D \ar[rrr]^{\phi}\ar[d]_{\phi}
&&&\tilde D \otimes D\ar[d]^{\Gamma \otimes id}\\
\tilde D \otimes D
\ar[rrr]_{id \otimes \phi}&&&\tilde D \otimes \tilde D\otimes D
}
}
Let $(t,s) \in D$:
\begin{eqnarray*}
(\Gamma \otimes id)\circ \phi (t,s) &=& (\Gamma \otimes id) \big(\sum_{s'\subseteq s} (t , s') \otimes(t/s', s/s')\big)\\
&=& \sum_{s'' \subseteq s'\subseteq s }  (t , s'') \otimes (t/s'' , s'/s'') \otimes (t/s', s/s').
\end{eqnarray*}
On the other hand, we have:
\begin{eqnarray*}
(id \otimes \phi)\circ \phi (t,s) &=& (id \otimes \phi) \big(\sum_{s''\subseteq s} (t , s'') \otimes(t/s'', s/s'')\big)\\
&=& \sum_{s'\subseteq s \atop \tilde s'\subseteq s/s''} (t , s'') \otimes (t/s'' , \tilde s') \otimes ((t/s'')/\tilde s', (s/s'')/\tilde s')\\
&=& \sum_{s'' \subseteq s'\subseteq s}  (t , s'') \otimes (t/s'' , s'/s'') \otimes (t/s', s/s').
\end{eqnarray*}
Then : $$(\Gamma \otimes id)\circ \phi  = (id \otimes \phi)\circ \phi,$$
and consequently $\phi$ is a coaction.
\end{proof}
\begin{remark}
We note here that $\phi (V) \subset \tilde V \otimes V$.   
\end{remark}

\section{Structures of associative algebras on the doubling spaces}
\subsection{Associative Product  on $V$}
Recall here that an element $(t, s)$ belongs to $V$ if $t$ is a tree and $s = P^c (t)$ i.e. $s$ is the pruning of the tree $t$ for an admissible cut $c$.
\begin{theorem}
Let $(t,s)$ and $(t',s')$ be two couples of forests such that $s = P^{c} (t)$ and $s' = P^{c'} (t')$. The product $\circledast : V \otimes V \longrightarrow V$ is defined by:
\begin{equation}
(t,s) \circledast (t',s') = \left\lbrace
\begin{array}{lcl}
(t,s \cup s')  \;\;\;\;\; \text{if} \;\; t' = R^c (t)\\
0  \;\; \;\; \;\; \;\; \text{if not},
\end{array}\right. 
\end{equation}
where $s \cup s'$ is the pruning of the cut $c'$ raised to the tree $t$, is associative.
\end{theorem}
\begin{proof}
Let $(t,s), (t',s')$ and $(t'',s'')$ be a three elements of $V$ i.e, there exist $c \in {\mop{Adm}(t)}, c' \in {\mop{Adm}(t')}$ and $c'' \in {\mop{Adm}(t'')}$ such that : $s = P^{c} (t), s' = P^{c'} (t')$ and $t'' = P^{c''} (t'')$.\\
We suppose firstly that $t' = R^c (t)$, otherwise the result is zero.
\begin{eqnarray*}
\big((t,s) \circledast (t',s')\big)\circledast (t'',s'') &=& (t,\tilde s' ) \circledast (t'',s'')\\
&=& (t,s \cup s' \cup s''), 
\end{eqnarray*}
where $\tilde s' = s \cup s'$ and $t'' = R^{c_{\tilde s'}} (t) = R^{c'} (R^{c}(t)) = R^{c'} (t')$. Then:
\begin{eqnarray*}
\big((t,s) \circledast (t',s')\big)\circledast (t'',s'') = \left\lbrace
\begin{array}{lcl}
(t,s \cup s' \cup s'')  \;\;\;\;\;  \text{if} \;\; t' = R^c (t) \;\text{and}\; t'' = R^{c'} (t')\\
0  \;\; \;\; \;\; \;\; \text{if not}.
\end{array}\right. 
\end{eqnarray*}
Otherwise, for $t'' = R^{c'} (t')$ we have: 
\begin{eqnarray*}
(t,s) \circledast \big( (t',s')\circledast (t'',s'')\big) &=& (t,s) \circledast (t', \tilde s'')\\
&=& (t,s \cup s' \cup s''), 
\end{eqnarray*}
where $\tilde s'' = s' \cup s''$ and $t' = R^{c} (t)$. Therefore:
\begin{eqnarray*}
(t,s) \circledast \big( (t',s')\circledast (t'',s'')\big) = \left\lbrace
\begin{array}{lcl}
(t,s \cup s' \cup s'')  \;\;\;\;\;  \text{if} \;\; t' = R^c (t)\; \text{and}\; t'' = R^{c'} (t')\\
0  \;\; \;\; \;\; \;\; \text{if not}.
\end{array}\right. 
\end{eqnarray*}
We therefore conclude that for all $(t,s), (t',s')$ and $(t'',s'')$ in $V$ we have:
$$\big((t,s) \circledast (t',s')\big)\circledast (t'',s'') = (t,s) \circledast \big( (t',s')\circledast (t'',s'')\big),$$
which proves the associativity  of the product $\circledast$.
\end{proof}

\subsection{Associative product on $\tilde V$}
Recall here that an element $(t, s)$ belongs to $\tilde V$ if $s$ is a subforest of the tree $t$.
\begin{theorem}\label{th5}
The product $ \sharp : \tilde V \otimes \tilde V \longrightarrow \tilde V$ defined by:
\begin{equation}
(t,s) \sharp (t',s') = \left\lbrace
\begin{array}{lcl}
(t,s \cup s')  \;\;\;\;\; \text{if} \;\; t' = t/s\\
0  \;\; \;\; \;\; \;\; \text{if not}
\end{array}\right. 
\end{equation}
is associative.
\end{theorem}
\begin{proof}
Let $(t,s), (t',s')$ and $(t'',s'')$ three elements of $\tilde V$, i.e, $s \subseteq t, s' \subseteq t' $ and $s'' \subseteq t''$\\
We suppose firstly that $t' = t/s$, if not, the result is zero.
\begin{eqnarray*}
\big((t,s) \sharp (t',s')\big)\sharp (t'',s'') &=& (t,\tilde s') \sharp (t'',s'')\\
&=& (t,s \cup s' \cup s''), 
\end{eqnarray*}
where $\tilde s' = s \cup s'$ and $t'' = t/\tilde s' = (t/s)/s' = t'/s'$. Therefore:
\begin{eqnarray*}
\big((t,s) \sharp (t',s')\big)\sharp (t'',s'') = \left\lbrace
\begin{array}{lcl}
(t,s \cup s' \cup s'')  \;\;\;\;\;  \text{if} \;\; t' = t/s \;\text{and}\;\; t'' = t'/s'\\
0  \;\; \;\; \;\; \;\; \text{if not}.
\end{array}\right. 
\end{eqnarray*}
Otherwise, for $t'' = t'/s'$ we have: 
\begin{eqnarray*}
(t,s) \sharp \big( (t',s')\sharp (t'',s'')\big) &=& (t,s) \sharp (t', \tilde s'')\\
&=& (t,s \cup s' \cup s''), 
\end{eqnarray*}
where $\tilde s'' = s' \cup s''$ and $t' = t'/s'$. Therefore:
\begin{eqnarray*}
(t,s) \sharp \big( (t',s')\sharp (t'',s'')\big) = \left\lbrace
\begin{array}{lcl}
(t,s \cup s' \cup s'')  \;\;\;\;\;  \text{if} \;\; t' = t/s \;\;\;\text{and}\;\;\; t'' = t'/s'\\
0  \;\; \;\; \;\; \;\; \text{if not}.
\end{array}\right. 
\end{eqnarray*}
We conclude that for all $(t,s), (t',s')$ and $(t'',s'')$ in $D$ we have:
$$\big((t,s) \sharp (t',s')\big) \sharp (t'',s'') = (t,s) \sharp \big( (t',s')\sharp (t'',s'')\big),$$
which proves the associativity of the product $\sharp$.
\end{proof}

\section{Relations between the laws on $V$ and $\tilde V$}
\begin{theorem}
The map: $ \psi : \tilde V \otimes  V  \longrightarrow  V$ defined by:
\begin{equation}
\psi \big((t,s) \otimes (u, P^c (u))\big) = \left\lbrace
\begin{array}{lcl}
(t,  P^{\tilde c}(t))  \;\;\;\;\; \text{if} \;\; u = t/s\\
0  \;\; \;\; \;\; \;\; \text{if not},
\end{array}\right. 
\end{equation}
where ${\tilde c}$ is the raising of $c$ to $t$, is an action of $\tilde V$ on $V$.
\end{theorem}
\begin{proof}
We have to verify the commutativity of this diagram:
\diagramme{
\xymatrix{\tilde V \otimes \tilde V\otimes V \ar[rrr]^{id \otimes \psi}\ar[d]_{\sharp \otimes id}
&&&\tilde V \otimes V\ar[d]^{\psi}\\
\tilde V \otimes V
\ar[rrr]_{\psi }&&& V
}
}
Let $(t,s)$ and $(t',s')$ be two elements of $\tilde V$, and $(u, P^c (u)) \in V$. We have:
\begin{eqnarray*}
(id \otimes \psi)\big((t,s)\otimes(t',s')\otimes (u, P^c (u))\big)  = \left\lbrace
\begin{array}{lcl}
(t,s)\otimes (t', P^{\bar c} (t')) \;\;  \text{if} \;\; u = t'/s'\\
0  \;\; \;\; \;\; \;\; \text{if not},
\end{array}\right. 
\end{eqnarray*}
 where ${\bar c}$ is the raising of  $c$ to $t'$. Then:
\begin{eqnarray*}
\psi \circ (id \otimes \psi)\big((t,s)\otimes (t',s')\otimes (u, P^c (u))\big) = \left\lbrace
\begin{array}{lcl}
\big(t, P^{\tilde {c}} (t)\big) \;\;  \text{if} \;\; t' = t/s \;\;\text{and} \;\; u = t'/s' \\
0  \;\; \;\; \;\; \;\; \text{if not},
\end{array}\right. 
\end{eqnarray*}
 where ${\tilde c}$ is the raising of $\bar c$ to $t$, i.e: ${\tilde c}$ is the raising of $c$ to $t$. Therefore:
\begin{eqnarray*}
\psi \circ (id \otimes \psi)\big((t,s)\otimes (t',s')\otimes (u, P^c (u))\big) = \left\lbrace
\begin{array}{lcl}
\big(t, P^{\tilde {c}} (t)\big) \;\;  \text{if} \;\; t' = t/s \;\;\text{and} \;\; u = t'/s'\\
0  \;\; \;\; \;\; \;\; \text{if not},
\end{array}\right. 
\end{eqnarray*}
where ${\tilde c}$ is the raising of $c$ to $t$.

On the other hand, we have:
\begin{eqnarray*}
(\sharp \otimes id)\big((t,s)\otimes (t',s')\otimes (u, P^c (u))\big)  = \left\lbrace
\begin{array}{lcl}
(t, s \cup s')\otimes (u, P^c (u)) \;\;  \text{if} \;\; t' = t/s\\
0  \;\; \;\; \;\; \;\; \text{if not}.
\end{array}\right.
\end{eqnarray*}
 Then:
\begin{eqnarray*}
\psi \circ (\sharp \otimes id)\big((t,s)\otimes (t',s')\otimes (u, P^c (u))\big)  = \left\lbrace
\begin{array}{lcl}
\big(t, P^{\tilde {c}} (t)\big) \;\;  \text{if} \;\; t' = t/s ,\;\;\;  u = t/(s \cup s')  = t'/s' \\
0  \;\; \;\; \;\; \;\; \text{if not},
\end{array}\right. 
\end{eqnarray*}
where ${\tilde c}$ is the raising of $c$ to $t$. We conclude then:
$$\psi \circ (id \otimes \psi) = \psi \circ (\sharp \otimes id),$$
which proves that $\psi$ is an action of $\tilde V$ on $V$.
\end{proof}
\begin{theorem}
The following diagram is commutative:
\diagramme{
\xymatrix{D \ar[rrr]^{\phi}\ar[d]_{\Delta}
&&&\tilde D \otimes D\ar[d]^{id\otimes \Delta}\\
D \otimes D
\ar[rrr]_{\phi \otimes id}&&&\tilde D \otimes D\otimes D
}
}
i.e:
\begin{equation}
(\phi\otimes id)\circ \Delta = (id\otimes \Delta)\circ \phi.
\end{equation}
Therefore, $\Delta$ is a comodule morphism from $(D, \Phi)$ to $(D\otimes D, \Phi\otimes id)$.
\end{theorem}
\begin{proof}
We have:
\begin{eqnarray*}
(id\otimes \Delta)\circ \phi (t,s) &=& (id\otimes \Delta) \big(\sum_{s'\subseteq s} (t , s') \otimes(t/s', s/s')\big)\\
&=& \sum_{s'\subseteq s \atop c \in {\mop{\tiny{Adm}}(s/s')}} (t , s') \otimes \big(t/s' , P^{c} (s/s')\big) \otimes \big( R^{c}(t/s'), R^{c}(s/s')\big).
\end{eqnarray*}
On the other hand, we have:
\begin{eqnarray*}
(\phi\otimes id)\circ \Delta (t,s) &=& (\phi\otimes id) \Big(\sum_{c' \in {\mop{\tiny{Adm}}(s)}}  \big(t, P^{c'}(s)\big)\otimes \big( R^{c'}(t), R^{c'}(s)\big)\Big) \\
&=& \sum_{c' \in {\mop{\tiny{Adm}}(s)} \atop s'\subseteq P^{c'}(s)} (t, s')\otimes \big(t/s' , P^{c'} (s) / s'\big) \otimes \big( R^{c'}(t), R^{c'}(s)\big).
\end{eqnarray*}
The conditions $\{ s'\subseteq P^{c'}(s)\; \text{and}\; c \in {\mop{Adm}(s/s')}\}$ and $\{c' \in {\mop{Adm}(s)}\;\text{and}\; s'\subseteq P^{c'}(s)\}$ are equivalent, where $c'$ is the raising of $c$. We then obtain the following equalities:\\
$P^{c'} (s) / s' = P^{c} (s/ s') \;;\; R^{c'}(t) = R^{c}(t/s')  \;\; and \;\; R^{c'}(s) = R^{c}(s/s')$, Which gives: 
\begin{eqnarray*}
(\phi\otimes id)\circ \Delta (t,s) &=& \sum_{c' \in {\mop{\tiny{Adm}}(s)} \atop s'\subseteq P^{c'}(s)} (t, s')\otimes \big(t/s' , P^{c'} (s) / s'\big) \otimes \big( R^{c'}(t), R^{c'}(s)\big) \\
&=& \sum_{s'\subseteq s \atop c \in {\mop{\tiny{Adm}}(s/s')}} (t , s') \otimes \big(t/s' , P^{c} (s/s')\big) \otimes \big( R^{c}(t/s'), R^{c}(s/s')\big).
\end{eqnarray*}
Therefore: $$(\phi\otimes id)\circ \Delta = (id\otimes \Delta)\circ \phi,$$
which proves the commutativity of the diagram.
\end{proof}
\begin{theorem}
The map $\psi$ verifies the following commutative diagram:
\diagramme{
\xymatrix{\tilde V \otimes  V\otimes V \ar[rrr]^{id \otimes \circledast}\ar[d]_{\psi \otimes id}
&&&\tilde V \otimes  V\ar[d]^{\psi}\\
V \otimes V
\ar[rrr]_{\circledast}&&&  V
}
}
i.e :
\begin{equation}
\psi \circ (id \otimes \circledast) = \circledast \circ (\psi \otimes id).
\end{equation}
\end{theorem}
\begin{proof}
Let $(u, P^c (u))$, $(u', P^{c'} (u'))$ be two elements of $V$ and $(t,s) \in \tilde V$, we have:
\begin{eqnarray*}
(id \otimes \circledast) \big((t,s)  \otimes(u, P^c (u))\otimes (u', P^{c'} (u'))\big) &=& \left\lbrace
\begin{array}{lcl}
(t,s)  \otimes \big(u, P^{c} (u) \cup P^{c'} (u')\big) \;\;  \text{if} \;\; u' = R^c (u)\\
0  \;\; \;\; \;\; \;\; \text{if not},
\end{array}\right.
\end{eqnarray*}
Then:
\begin{eqnarray*}
(id \otimes \circledast) \big((t,s)  \otimes(u, P^c (u))\otimes (u', P^{c'} (u'))\big)
&=& \left\lbrace\begin{array}{lcl}
(t,s)  \otimes \big(u, P^{\bar c'} (u) \big) \;\;  \text{if} \;\; u' = R^c (u)\\
0  \;\; \;\; \;\; \;\; \text{if not},
\end{array}\right.
\end{eqnarray*}
where ${\bar c'}$ is the raising of $c'$ to $u$. Then:
\begin{eqnarray*}
\psi \circ (id \otimes \circledast) \big((t, s) \otimes (u, P^c (u))\otimes (u', P^{c'} (u'))\big)  &=& \left\lbrace
\begin{array}{lcl}
\big(t, P^{\tilde c'} (t) \big) \;\;  \text{if} \;\; u = t/s, \;\; \; u' = R^{c} (u)\\
0  \;\; \;\; \;\; \;\; \text{if not},
\end{array}\right.
\end{eqnarray*}
where ${\tilde c'}$ is the raising of $\bar c'$ to $t$. The cut ${\tilde c'}$ can be also seen as raising of $c'$ to $t$. Therefore:   
\begin{eqnarray*}
\psi \circ (id \otimes \circledast) \big((t, s) \otimes (u, P^c (u))\otimes (u', P^{c'} (u'))\big)  &=& \left\lbrace
\begin{array}{lcl}
\big(t, P^{\tilde c'} (t) \big) \;\;  \text{if} \;\; u = t/s,\;\;\; u' = R^{c} (u)\\
0  \;\; \;\; \;\; \;\; \text{if not},
\end{array}\right.
\end{eqnarray*}
where ${\tilde c'}$ is the raising of $c'$ to $t$.

On the other hand, we have:
\begin{eqnarray*}
(\psi \otimes id) \big((t, s)\otimes (u, P^{c} (u))\otimes (u', P^{c'} (u'))\big) = \left\lbrace
\begin{array}{lcl}
\big(u, P^{\tilde c} (t) \big) \otimes \big(u', P^{c'} (u')\big) \;\;  \text{if} \;\; u = t/s\\
0  \;\; \;\; \;\; \;\; \text{if not},
\end{array}\right.
\end{eqnarray*}
where ${\tilde c}$ is the raising of $c$ to $t$. Then:
\begin{eqnarray*}
\circledast \circ(\psi \otimes id) \big((t, s)\otimes (u, P^{c} (u)) \otimes (u', P^{c'} (u')\big)
&=& \left\lbrace\begin{array}{lcl}
\big(t, P^{\tilde c} (u) \cup P^{c'} (u')  \big) \;\;  \text{if} \;\; {u = t/s, \atop u' = R^{\tilde c} (t)}\\
0  \;\; \;\; \;\; \;\; \text{if not}
\end{array}\right.\\
&=& \left\lbrace\begin{array}{lcl}
\big(t, P^{\tilde c'} (t) \big) \;\;  \text{if} \;\;\; {u = t/s,  \atop u' = R^{\tilde c} (t) = R^{c} (u)}\\
0  \;\; \;\; \;\; \;\; \text{if not},
\end{array}\right.
\end{eqnarray*}
where ${\tilde c'}$ is the raising of $c'$ to $t$. Therefore:
$$\psi \circ (id \otimes \circledast) = \circledast \circ (\psi \otimes id),$$
and consequently the diagram is commutative.
\end{proof}
\begin{definition}
Let $\xi : \tilde V \otimes \tilde V \otimes V \longrightarrow \tilde V \otimes V$ be the map defined by:
\begin{enumerate}
\item $\xi \big( (t', s') \otimes (t'',s'')\otimes (u, v)\big) = (t', s' \cup s'')\otimes (t'/(s' \cup s''), v),$\\
if $t'' = R^c(t')$, $v = P^{\tilde{c}}(t')$ and $u = t'/s'$, where $c$ is an admissible cut of $t'$, and $\tilde{c}$ is an admissible cut of $t'$ does not meet $s'$ and $s''$.
\item $\xi \big( (t', s') \otimes (t'',s'')\otimes (u, v)\big) = 0,$\\ 
if $t', t'', s', s'', u$ and $v$ are forests which do not match the conditions of item (1). 
\end{enumerate} 
\end{definition}
\begin{theorem}
The two maps $\phi$ and $\xi$ make the following diagram commute:
\diagramme{
\xymatrix{
V \ar[d]_{\Delta} \ar[rr]^{\phi} 
&&\tilde V \otimes V  \ar[d]^{id \otimes \Delta}\\
 V \otimes V \ar[d]_{\phi\otimes \phi}&&\tilde V \otimes V \otimes V \\
\tilde V \otimes V \otimes \tilde V \otimes V \ar[rr]_{\tau^{23}} 
&&\ar[u]_{\xi \otimes id } \tilde V \otimes \tilde V \otimes  V \otimes V
 }
}
i.e :
\begin{equation}
(id \otimes \Delta) \circ \phi = (\xi \otimes id ) \circ \tau^{23} \circ (\phi\otimes \phi) \circ \Delta. 
\end{equation}
\end{theorem}
\begin{proof}
We have:
\begin{eqnarray*}
(id\otimes \Delta)\circ \phi (u,v) &=& (id\otimes \Delta) \big(\sum_{s\subseteq v} (u , s) \otimes(u/s, v/s)\big)\\
&=& \sum_{s\subseteq v \atop c \in {\mop{\tiny{Adm}}(v/s)}} (u , s) \otimes \big(u/s , P^{c} (v/s)\big) \otimes \big( R^{c}(u/s), R^{c}(v/s)\big).
\end{eqnarray*}
On the other hand, we have:
\begin{eqnarray*}
&&(\xi \otimes id ) \circ \tau^{23} \circ (\phi\otimes \phi) \circ \Delta (u,v) = (\xi \otimes id ) \circ \tau^{23}\Big(\sum_{c \in {\mop{\tiny{Adm}}(v)}}  \phi \big(u, P^{c}(v)\big)\otimes \phi\big( R^{c}(u), R^{c}(v)\big)\Big)  \\
&&= (\xi \otimes id ) \circ \tau^{23} \Big(\sum_{c \in {\mop{\tiny{Adm}}(v)}} \sum_{s' \subseteq P^{c} (v) \atop s' \subseteq R^{c} (v)}\hskip-0.2cm(u, s')\otimes \big(u/s', P^{c} (v)/s'\big) \otimes  \big(R^{c}(u), s''\big) \otimes \big(R^{c}(u)/s'' , R^{c}(v)/s'' \big)\Big)\\
&&= (\xi \otimes id )\Big(\sum_{c \in {\mop{\tiny{Adm}}(v)}} \sum_{s' \subseteq P^{c} (v) \atop s'' \subseteq R^{c} (v)}(u, s')\otimes  \big(R^{c}(u), s''\big) \otimes  \big(u/s', P^{c} (v)/s'\big)\otimes \big(R^{c}(u)/s'' , R^{c}(v)/s'' \big)\Big)\\
&&= \sum_{c \in {\mop{\tiny{Adm}}(v)}} \sum_{s' \subseteq P^{c} (v) \atop s''\subseteq R^{c} (v)}(u, s' \cup s'')\otimes  \Big(u\big\slash s'\cup s'', P^{c} (v)\Big\slash\ (s'\cup s'')\cap P^{c} (v)\Big) \otimes \big(R^{c}(u)/s'' , R^{c}(v)/s'' \big)\\
&&=\sum_{c \in {\mop{\tiny{Adm}}(v)}} \sum_{s \subseteq v \atop \text{containing no edge of c}}\hskip-0.8cm(u, s)\otimes  \big(u/s, P^{c} (v)\big\slash\ s\cap P^{c} (v)\big) \otimes \big(R^{c}(u)\big\slash\ s\cap R^{c}(u) , R^{c}(v)\big\slash\ s\cap R^{c}(v) \big)\\
&&= \sum_{s \subseteq v/s} \sum_{c \in {\mop{\tiny{Adm}}(v/s)}}(u, s)\otimes  \big(u/s, P^{c} (v/s)\big) \otimes \big(R^{c}(u/s) , R^{c}(v/s) \big).
\end{eqnarray*}
Hence: $(id\otimes \Delta)\circ \phi = (\xi \otimes id ) \circ \tau^{23} \circ (\phi\otimes \phi) \circ \Delta$, which proves the theorem. 
\end{proof}
\begin{remark}
We notice here that, this diagram extends to the commutative diagram:
\diagramme{
\xymatrix{
D \ar[d]_{\Delta} \ar[rr]^{\phi} 
&&\tilde D \otimes D  \ar[d]^{id \otimes \Delta}\\
 D \otimes D \ar[d]_{\phi\otimes \phi}&&\tilde D \otimes D \otimes D \\
\tilde D \otimes D \otimes \tilde D \otimes D \ar[rr]_{\tau^{23}} 
&&\ar[u]_{\xi \otimes id } \tilde D \otimes \tilde D \otimes  D \otimes D
 }
}
where the arrows are now algebra morphisms.
\end{remark}



\begin{thebibliography}{abcdsfgh}
{\small{
\bibitem{MBM}  M. Belhaj Mohamed, \textsl{Renormalisation dans les alg\`ebres de Hopf gradu\'ees connexes}, PhD th\`ese (2014).
\bibitem{mbm} M. Belhaj Mohamed, \textsl{Doubling bialgebras of graphs and Feynman rules}, Confluentes Mathematici, to appear (2016).
\bibitem{DMB} M. Belhaj Mohamed, D. Manchon, \textsl{Bialgebra of specified graphs and external structures}, Ann. Inst. Henri Poincar\'e, D, Volume 1, Issue 3, pp. 307-335 (2014).
\bibitem{bu}J. C. Butcher, \textsl{An algebraic theory of integration methods}, Math. Comp. 26 (1972) 79–106.
\bibitem{ckm} D. Calaque, K. Ebrahimi-Fard, D. Manchon, \textsl{ Two interacting Hopf algebras of trees: a Hopf-algebraic approach to composition and substitution of B-series}, Advances in Applied Mathematics, 47, $n^{\circ} 2$, 282-308 (2011).
\bibitem{ca} A. Cayley, \textsl{A theorem on trees}, Quart. J. Math. 23 (1889) 376-378.
\bibitem{ad98} A. Connes, D. Kreimer, \textsl{Hopf algebras, Renormalization and Noncommutative Geometry}, Comm. in Math. Phys. 199, 203-242 (1998).
\bibitem{A.D2000} A. Connes, D. Kreimer, \textsl{Renormalization in quantum field theory and the Riemann-Hilbert problem. I. The Hopf algebra structure of graphs and the main theorem}, Comm. Math. Phys. 210, $n^{\circ} 1$, 249-273 (2000).
\bibitem{lf} L. Foissy, \textsl{ Les alg\`ebres de Hopf des arbres enracin\'es d\'ecor\'es I + II}, th\`ese, Univ. de Reims (2002), et Bull. Sci. Math. {\bf{126}}, $n^{\circ} 3$, 193--239 et $n^{\circ} 4$, 249--288 (2002).
\bibitem{gl} R. Grossman, R. G. Larson, \textsl{Hopf-algebraic structure of families of trees}, J. Algebra 126 (1989) 184-210.
\bibitem{DK98} D. Kreimer, \textsl{On the Hopf algebra structure of perturbative quantum field theories}, Adv. Theor. Math. Phys. 2 (1998).
\bibitem{Dm11} D. Manchon, \textsl{On bialgebra and Hopf algebra of oriented graphs}, Confluentes Math. Volume 04, $n^{\circ} 1$ (2012).
\bibitem{ms} D. Manchon, A. Sa\"idi, {\textsl{Lois pr\'e-Lie en interaction}}, Comm. Alg. vol 39, $n^{\circ} 10$, 3662-3680 (2011).
\bibitem{mw} H. Munthe-–Kaas, W.Wright, \textsl{On the Hopf Algebraic Structure of Lie Group Integrators}, Found. Comput. Math. 8 (2008) 227-257.
}}
\end{thebibliography}
\end{document}